\documentclass[hidelinks,onefignum,onetabnum]{siamart220329}
\usepackage{macros_thesis}
\usepackage{todonotes}

\usepackage{physics}
\usepackage{multicol}
\usepackage{subcaption}
\usepackage{siunitx}
\usepackage{amsfonts}
\usepackage{graphicx}
\usepackage{epstopdf}
\usepackage{algorithmic}

\ifpdf
  \DeclareGraphicsExtensions{.eps,.pdf,.png,.jpg}
\else
  \DeclareGraphicsExtensions{.eps}
\fi

\newsiamremark{remark}{Remark}
\newsiamremark{hypothesis}{Hypothesis}
\crefname{hypothesis}{Hypothesis}{Hypotheses}
\newsiamthm{claim}{Claim}

\headers{}{}

\title{Improving the stability and efficiency of high-order
  operator-splitting methods\thanks{Submitted to the editors
    DATE.\funding{This work was funded by the Natural Sciences and
      Engineering Research Council of Canada under its Discovery Grant
      Program (RGPN-2020-04467) and the National Oceanic and
      Atmospheric Administration (NOAA), awarded to the Cooperative
      Institute for Research on Hydrology (CIROH) through the NOAA
      Cooperative Agreement with The University of Alabama,
      NA22NWS4320003 [RJS].)}}}

\author{Siqi Wei\thanks{Department of Mathematics and Statistics, University of Saskatchewan, Saskatoon, SK, Canada
		(\email{siqi.wei@usask.ca}).}
\and Victoria Guenter\thanks{Department of Computer Science, University of Saskatchewan, Saskatoon, SK, Canada
  (\email{v.guenter@usask.ca}).}
\and Raymond J. Spiteri\thanks{Department of Computer Science, University of Saskatchewan, Saskatoon, SK, Canada
	(\email{spiteri@cs.usask.ca}).}
}

\usepackage{amsopn}

\ifpdf
\hypersetup{
  pdftitle={Title},
  pdfauthor={A.B}
}
\fi

\theoremstyle{plain}
\newtheorem{thm}{Theorem}[section]

\newtheorem{rmk}{Remark}[section]

\begin{document}

\maketitle

\begin{abstract}
  Operator-splitting methods are widely used to solve differential
  equations, especially those that arise from multi-scale or
  multi-physics models, because a monolithic (single-method) approach
  may be inefficient or even infeasible. The most common
  operator-splitting methods are the first-order Lie--Trotter (or
  Godunov) and the second-order Strang (Strang--Marchuk) splitting
  methods. High-order splitting methods with real coefficients require
  backward-in-time integration in each operator and hence may be
  adversely impacted by instability for certain operators such as
  diffusion. However, besides the method coefficients, there are many
  other ancillary aspects to an overall operator-splitting method that
  are important but often overlooked. For example, the operator
  ordering and the choice of sub-integration methods can significantly
  affect the stability and efficiency of an operator-splitting
  method. In this paper, we investigate some design principles for the
  construction of operator-splitting methods, including minimization
  of local error measure, choice of sub-integration method,
  maximization of linear stability, and minimization of overall
  computational cost. We propose a new four-stage, third-order,
  2-split operator-splitting method with seven sub-integrations per
  step and optimized linear stability for a benchmark problem from
  cardiac electrophysiology. We then propose a general principle to
  further improve stability and efficiency of such operator-splitting
  methods by using low-order, explicit sub-integrators for unstable
  sub-integrations.  We demonstrate an almost 30\% improvement in the
  performance of methods derived from these design principles compared
  to the best-known third-order methods.
\end{abstract}

\begin{keywords}
  operator splitting, fractional-step methods, Runge--Kutta methods,
  linear stability analysis, cardiac electrophysiology simulation
\end{keywords}

\begin{MSCcodes}
65L05, 65L06, 65L20
\end{MSCcodes}


\section{Introduction}
\label{sec:beat_ruth_intro}

Operator splitting is a popular approach to solving multi-physics
problems.  This approach splits the original differential equation
into several operators, integrates each operator with an appropriate
method, and composes the solutions from the sub-integrations to obtain
a solution to the original problem. Generally speaking, any approach
that involves distinct treatments of the operators, such as additive,
partitioned, or implicit-explicit (IMEX)
methods~\cite{Yanenko1971,ars1997,kennedy2003,sandu2015,spiteri_wei_FSRK},
can be considered to be a form of operator splitting. In this paper,
we focus on compositional operator-splitting
methods~\cite{hairer2006}, also referred to as fractional-step
methods~\cite{Yanenko1971,Tyson2000,spiteri_wei_FSRK}(see
also~\cref{eq:s-stageOS} below).

A fundamental difference between fractional-step methods and
implicit-explicit (IMEX) methods or other additive methods is that in
the case of fractional-step methods, each subsystem (fractional step)
is integrated independently. The coupling between subsystems occurs
only through the initial conditions used for the sub-integrations. In
the case of additive methods, the sub-integrations have a stronger
coupling. Fractional-step methods are in fact a special case of
additive methods~\cite{spiteri_wei_FSRK}.

Consider a $2$-additively split ordinary differential equation (ODE)
\begin{equation}
	\label{eq:2splitode}
	\dv{\yub}{t} = \mathcal{F}(t,\yub) = \Fub{1}(t,\yub) + \Fub{2}(t,\yub).
\end{equation}
Let
$\displaystyle \DOp{\ell} = \sum\limits_{j} \Fub{\ell}_j(\yub)
\pdv{}{y_j} $ be the Lie derivative of $\Fub{\ell}$.  Let
$\varphi_t^{[\ell]}$ be the exact flow of
\begin{equation}
  \label{eq:subsys} \displaystyle \dv{\yub^{[\ell]}}{t}= \Fub{\ell}(t,\yub^{[\ell]}),\quad  \ell=1,2. 
\end{equation}
It can be shown that
$\varphi_t^{[\ell]} (\yub_0) = \exp(\Dt \DOp{\ell}) \text{Id}(\yub_0)$
and the exact solution of the original differential equation
\cref{eq:2splitode} is
$\Phi_{\Dt}(\yub_0) = \exp(\left(\DOp{1} + \DOp{2}\right)\Dt)
\text{Id}(\yub_0)$ \cite{hairer2006}.

A popular method to approximate $\Phi_{\Dt}(\yub_0)$ by solving the
subsystems~\cref{eq:subsys} is
\begin{equation} \label{eq:1storderapprox} \Phi_{\Dt}(\yub_0) =
  \exp(\left(\DOp{1}+\DOp{2}\right)\Dt) \text{Id}(\yub_0) \approx
  \exp(\DOp{2}\Dt) \exp(\DOp{1}\Dt) \text{Id}(\yub_0).
\end{equation}
This method can be written as
\begin{equation}
	\label{eq:LT}
	\Psi_{\Dt}^{\text{LT}} =  \varphi_{\Dt}^{[2]} \circ \varphi_{\Dt}^{[1]}
\end{equation}
and is known as the Lie--Trotter (or Godunov) method
\cite{Trotter1958,Godunov1959}.

Another highly popular method operator-splitting method is the
Strang (Strang--Marchuk)
method~\cite{Strang1968,Marchuk1971} defined as
\begin{equation}
  \label{eq:Strang}
  \Psi_{\Dt}^{\text{S}} = \varphi_{\Dt/2}^{[1]} \circ \varphi_{\Dt}^{[2]} \circ \varphi_{\Dt/2}^{[1]}.
\end{equation}

Given a set of coefficients
$\{\aaalpha{k}{\ell}\}_{k=1,2,\dots, s}^{\ell=1,2}$, we define a
general $s$-stage, 2-split fractional-step method for solving
\cref{eq:2splitode} as
\begin{equation}
  \label{eq:s-stageOS}
  \Psi_{\Dt} = \varphi_{\aaalpha{s}{2}\Dt}^{[2]} \circ
  \varphi_{\aaalpha{s}{1}\Dt}^{[1]} \circ
  \varphi_{\aaalpha{s-1}{2}\Dt}^{[2]} \circ
  \varphi_{\aaalpha{s-1}{1}\Dt}^{[1]} \circ \cdots \circ  \varphi_{\aaalpha{1}{2}\Dt}^{[2]} \circ \varphi_{\aaalpha{1}{1}\Dt}^{[1]},
\end{equation}
where a \emph{stage} is defined as a set of sub-integrations through the
complete sequence of operators, in this case $\{1,2\}$. In other
words, at each stage $k$, the subsystems \cref{eq:subsys} are
integrated over a time step $\aaalpha{k}{\ell} \Dt$, for $\ell = 1,2$,
using the solution from the previous sub-integration as the initial
condition.

The integration of fractions of the right-hand side generally
introduces a splitting error that is in addition to truncation errors
from the sub-integrations or roundoff errors. For example, the
Lie--Trotter method~\cref{eq:LT} is first-order accurate in the sense
that the splitting error is $\bigOh(\Dt)$. The Strang
method~\cref{eq:Strang} is second-order.

In general, one way to construct high-order methods is by choosing the
coefficients $\aaalpha{k}{\ell}$ satisfy order
conditions~\cite{Hansen2009,auzinger2014,auzinger2016practical}. Another
way is by composing low-order methods
\cite{yoshida1990,hairer2006,Hansen2009}.  Once a method is found for
a given order, its adjoint, i.e., the inverse map with the time
direction reversed, also has the same order~\cite[Theorem
3.1]{hairer2006}. When discussing the existence of a given method
below, we omit explicitly mentioning the adjoint as a distinct method.

Operator-splitting methods of order three or greater, however, require
backward-in-time integration for each
operator~\cite{goldman1996}. These backward integrations can be
unstable for certain types of operators, e.g., diffusion, that are not
time-reversible. This instability has undoubtedly been a deterrent to
trying high-order methods in applications for systems that are not
time-reversible. Nonetheless, stability and convergence are possible
despite the presence of some
instability~\cite{cervi2018,spiteri_wei_FSRK,Wei2020}.

There are many other ancillary considerations as well that go into the
overall design, implementation, and ultimately performance of an
operator-splitting method. For example, one must also select the
sub-integration methods. One can use the exact solution of a
subsystem~\cref{eq:subsys}, if available (and desirable), or
approximate the solution using a numerical method. In this paper, we
focus on the family of operator-splitting methods that uses
Runge--Kutta methods as sub-integrators.  These operator-splitting
methods are fractional-step Runge--Kutta (FSRK)
methods~\cite{spiteri_wei_FSRK}.  The stability function of an
$N$-split FSRK method is presented in \cite{spiteri_wei_FSRK}; here,
we present the special case of $N=2$ operators.

Let $\stabfunc{k}{\ell} (z)$ be the stability function of
the Runge--Kutta method applied to operator $\Fub{\ell}$ at stage $k$,
the stability function of the resulting FSRK method is
\begin{equation}
	\label{eq:fsrk_stab}
	R\left(z^{[1]}, z^{[2]}\right) = \prod\limits_{k=1}^s \stabfunc{k}{1} \left(\aaalpha{k}{1} z^{[1]}\right) \stabfunc{k}{2} \left(\aaalpha{k}{2} z^{[2]}\right). 
\end{equation}  
\Cref{eq:fsrk_stab} shows that the stability function of the FSRK
method is the product of the stability functions of each Runge--Kutta
method with argument, however, scaled by the operator-splitting method
coefficient $\aaalpha{k}{\ell}$. This interpretation generally holds for
$N$ operators. \Cref{eq:fsrk_stab} implies that the stability of an
FSRK method is generally affected by the ordering of the operators
because the arguments are operator-dependent, as well by the method
coefficients $\aaalpha{k}{\ell}$ themselves and the choice of
sub-integration methods.
 
Although it is clear that different operator-splitting methods perform
differently, it is less clear exactly how performance is impacted by
the choice of the method coefficients. It is shown in
\cite{spiteri2023_3split} that methods with a good balance between
accuracy and cost per step are more efficient, especially when the
splitting error is the dominant source of error.  Also when using
operator-splitting methods, it is also unclear what effect the
ordering of the operators has on the performance of the overall
method. For example, changing the ordering of the operators in the
Lie--Trotter method does not affect the linear stability according
to~\cref{eq:fsrk_stab}, nor does it generally affect the overall
computation time per step because each operator is integrated over the
same step-size $\Delta t$; however, the splitting errors are generally
different (albeit trivially in this case).  It is also straightforward
to see that changing the ordering of operators for the Strang method
can make a difference in computational efficiency if one operator is
significantly more expensive than the
other(s)~\cite{spiteri2023_3split}. For $2$-split palindromic methods,
changing the order of operators makes no difference in stability;
however, the accuracy can vary. In general, operator ordering can
significantly impact the stability, accuracy, and performance of an
operator-splitting method.


In this paper, we propose two strategies to improve the stability and
efficiency of an FSRK method: first by optimizing the linear stability
region using the method coefficients and operator ordering and second
by using inexpensive explicit methods for unstable sub-integrations to
improve both stability and computational efficiency per step while
maintaining acceptable increases in error.

The remainder of the paper is organized as follows. In
\cref{sec:beat_ruth_background}, we give an introduction to the order
conditions of an $s$-stage operator-splitting method, the local error
measure (LEM), which can be used as a metric to compare methods of the
same order, some results on methods with minimal LEM, and the linear
stability function of the FSKR method. In \cref{sec:beat_ruth_main},
we construct new operator-splitting methods by
optimizing 
the intersection of the linear stability region with the negative real
axis. 
In \cref{sec:beat_ruth_num_exp}, we use the Niederer benchmark problem
from cardiac electrophysiology~\cite{niederer2011} to compare the
performance of the newly constructed methods with other methods such
as the classical third-order Ruth method and demonstrate how the two
strategies improve the stability and efficiency of FSRK methods. In
\cref{sec:beat_ruth_conclu}, we summarize the findings of this study.


\section{Background}
\label{sec:beat_ruth_background}

The purpose of this study is to describe design principles that lead
to new high-order operator-splitting methods that outperform existing
methods.  For a given order of accuracy, a general approach to design
time-stepping methods is to minimize the leading error
term~\cite{HairerNorsettWanner1993}, which for operator-splitting
methods is through minimization of the local error
measure~\cite{auzinger2016practical}. Another approach to design
time-stepping methods is to optimize stability, e.g.,~\cite{ars1997,
  kennedy2003}. For this purpose, we specialize the methods considered
to FSRK methods. We then use linear stability analysis to maximize
stability through the ordering of the operators as well as the choice
of the method coefficients to maximize the extent along the negative
real axis of the linear stability region. Finally, we further improve
stability and computational efficiency per step through the choice of
sub-integrator.

\subsection{Order conditions of operator-splitting methods}
\label{subsec:order_cond}

The order of accuracy is an important characteristic of a numerical
method and can often be determined by satisfying a set of order
conditions.  The order conditions for operator-splitting methods can
be derived from the Baker--Campbell--Hausdorff
formula~\cite{hairer2006,blanes2024}.  The following are the order
conditions at order $p$, $p=1,2,3,4$, for the method
\cref{eq:s-stageOS} with coefficients
$\{\aaalpha{k}{\ell}\}_{k=1,2,\dots, s}^{\ell=1,2}$:

\begin{subequations}
  \label{orders123}
 	\begin{align} 
		\label{order1}
		p &= 1:  \qquad \sum\limits_{k=1}^s \alpha_k^{[1]} = 1, \quad
		\sum\limits_{k=1}^s \alpha_k^{[2]} = 1,  \\ 
		\label{order2}
		p &= 2:  \qquad \sum\limits_{i=1}^s \alpha_i^{[2]}\left(
		\sum\limits_{k=1}^ i \alpha_{k}^{[1]}\right) = \frac{1}{2},  \\ 
		\label{order3}
      p &= 3: \qquad
          \sum\limits_{i=2}^{s}
          \alpha_i^{[1]}\left(\sum\limits_{k=1}^{i-1} \alpha_k^{[2]}\right)^2
          = \frac{1}{3},  \quad
          \sum\limits_{i=1}^s \alpha_i^{[1]} \left(\sum\limits_{k=i}^s \alpha_k^{[2]}\right)^2 = \frac{1}{3}, \\
      p &= 4: \qquad 
          \sum\limits_{i=1}^{s-1}
          \alpha_i^{[2]}\left(\sum\limits_{k=i+1}^s \alpha_k^{[1]}\right)^3
          = \frac{1}{4}, \nonumber  \\
          &  \qquad \sum\limits_{i=1}^{s-1} (\alpha_i^{[2]})^2 \left(\sum\limits_{k=i+1}^s \alpha_k^{[1]}\right)^2 + 2\,\sum\limits_{i=1}^{s-2} \alpha_i^{[2]}\left(\sum\limits_{k=i+1}^{s-1} \alpha_k^{[2]} \left(\sum\limits_{l=k+1}^{s} \alpha_l^{[1]} \right)^2\right)  = \frac{1}{6}, \nonumber \\ 
          & \qquad \qquad \ \sum\limits_{i=2}^{s}
            \alpha_i^{[1]}\left(\sum\limits_{k=1}^{i-1} \alpha_k^{[2]}\right)^3
            = \frac{1}{4}. \nonumber
	\end{align}  
\end{subequations} 

As mentioned, operator-splitting methods are typically constructed by
either solving the order conditions for a given order or by
compositions of lower-order methods (and possibly their adjoints) to
get higher-order methods. For example, the second-order Strang
method~\cite{Strang1968} is a composition of the first-order
Lie--Trotter method and its adjoint over $\Dt/2$.  The third-order
Ruth method~\cite{ruth1983} is derived by solving the order conditions
for $p=1,2,3$. The coefficients of the Ruth method as presented
in~\cite{ruth1983} are given in \cref{tab:Ruthcoeff}.
\begin{table}[htbp]
  \centering
  \caption{Coefficients $\alpha_k^{[\ell]}$ for the Ruth method}
  \begin{tabular}{|c|c|c|}
    \hline 
    $k$ & $\displaystyle \alpha_k^{[1]}$ &  $\displaystyle \alpha_k^{[2]}$  \\
    \hline 
    1 & $7/24$  & $2/3$ \\
    \hline 
    2 & $3/4$  & $-2/3$\\
    \hline 
    3 & $-1/24$  & $1$ \\
    \hline 
  \end{tabular}
  \label{tab:Ruthcoeff}
\end{table}
The Ruth method is a particularly elegant \emph{isolated} solution of
the order conditions $p=1,2,3$. There are also in fact infinitely many
2-split, three-stage, third-order operator-splitting methods
consisting of two one-parameter families~\cite{Hansen2009}. The Ruth
method requires six sub-integrations, and it turns out it is not
possible to achieve third order with fewer sub-integrations. We
comment further on these aspects in~\cref{subsec:optimalLEM}.

\subsection{Local error measure of operator-splitting methods}
\label{subsec:lem}

One way to measure the accuracy of an operator-splitting method is via
the local error measure proposed in \cite{auzinger2016practical}.  It
is shown in \cite{auzinger2016practical} that for a method of order
$p$, the leading local error $\localerror_{p+1}$ can be written as
\begin{equation}
	\label{eq:localerror}
	\localerror_{p+1} = \frac{\Dt^{p+1}}{(p+1)!}\sum\limits_{j=1}^{\gamma_{p+1}} \comcoeff{p+1}{j} \commu{p+1}{j},
\end{equation}
where $\comcoeff{p+1}{j}$ are constants that depend on the splitting
coefficients $\aaalpha{k}{\ell}$ and $\commu{p+1}{j}$ are the
$\gamma_{p+1}$ commutators of $\DOp{1}$ and $\DOp{2}$ of order
$(p+1)$. For example, for an $s$-stage method, the leading local error
$\localerror_{3}$ is
\begin{align*}
	\localerror_3 = & \frac{\Dt^3}{3!} \left\{\left( 3\left(\sum\limits_{i=1}^{s-1} \aaalpha{i}{2}\left(\sum\limits_{k=i+1}^s \aaalpha{k}{1} \right)^2 \right)- 1\right) [\DOp{1},[\DOp{1},\DOp{2}]] \right. \\
	& \left. + \left( 3\left(\sum\limits_{i=1}^{s} \aaalpha{i}{1}\left(\sum\limits_{k=i}^s \aaalpha{k}{2} \right)^2 \right)- 1\right) [[\DOp{1},\DOp{2}],\DOp{2}] \right\}.
\end{align*}
To compare the accuracy of different methods of order $p$, it is
reasonable to consider
\begin{equation*}
	\left( \sum\limits_{j=1}^{\gamma_{p+1}} |\comcoeff{p+1}{j}|^2 \right)^{1/2}. 
\end{equation*}
However, Auzinger et al. \cite{auzinger2016practical} proposed to use
the measure
\begin{equation}
	\label{eq:lem}
	\text{LEM} (p) := \left( \sum\limits_{j=1}^{\gamma_{p+1}} |\lambda_{p+1,j}|^2 \right)^{1/2}, 
\end{equation}
where $\lambda_{p+1,j}$ are coefficients of leading monomials in the
sense of lexicographical order in the expanded commutators. These
leading monomials correspond to the Lyndon words over the alphabet
$\DOp{1}$ and $\DOp{2}$. The advantages of using $\lambda_{p+1,j}$
instead of $\comcoeff{p+1}{j}$ are that the coefficients
$\lambda_{p+1,j}$ are constructed when finding order conditions for
order $p+1$ and they are easier to compute than
$\comcoeff{p+1}{j}$. More details regarding the construction of order
conditions can be found in~\cite{auzinger2016practical}. Because we
focus on third-order methods in this study, we give the LEM of
third-order method explicitly:
\begin{equation}
	\label{eq:lem_3rd}
	\text{LEM}(3)= \left(|\lambda_{4,1}|^2 + |\lambda_{4,2}|^2 + |\lambda_{4,3}|^2\right)^{1/2},
\end{equation}
where 
\begin{align*}
	\lambda_{4,1} &= 4\left[ \sum\limits_{i=1}^{s-1}
	\alpha_i^{[2]}\left(\sum\limits_{k=i+1}^s \alpha_k^{[1]}\right)^3 \right] -1, \\ 
	\lambda_{4,2} &= 6\left[ \sum\limits_{i=1}^{s-1} \left(\alpha_i^{[2]}\right)^2 \left(\sum\limits_{k=i+1}^s \alpha_k^{[1]}\right)^2 + 2\sum\limits_{i=1}^{s-2} \alpha_i^{[2]}\left(\sum\limits_{k=i+1}^{s-1} \alpha_k^{[2]} \left(\sum\limits_{l=k+1}^{s} \alpha_l^{[1]} \right)^2\right)  \right] - 1,  \\ 
	\lambda_{4,3} &= 4\left[  \sum\limits_{i=2}^{s}
	\alpha_i^{[1]}\left(\sum\limits_{k=1}^{i-1} \alpha_k^{[2]}\right)^3  \right] -1.  \\ 
\end{align*}

\begin{rmk}
  We note that the LEM is generally a fairly coarse measure of the
  splitting error. For example, it does not consider the commutators,
  which are distinctly problem-dependent and may vary greatly in
  magnitude.  Moreover, the LEM does not take into account the
  ordering of the operators or the sub-integration methods. These
  factors often have a large impact on the overall performance of an
  operator-splitting method. However, if each
  $\varphi_{\aaalpha{k}{\ell}\Dt}^{[\ell]}$ in~\cref{eq:s-stageOS} is
  approximated using a Runge--Kutta method, the resulting method can
  be written as fractional-step Runge--Kutta method, which can in turn
  be represented as an additive Runge--Kutta (ARK)
  method~\cite{spiteri_wei_FSRK}.  The ARK representation takes into
  account the operator-splitting method, the Runge--Kutta methods used
  for the sub-integrations, and the order in which the operators are
  applied. As shown in~\cite{ketcheson2023}, one can use the
  generalized $B$-series of an ARK method to study the accuracy of the
  ARK method more precisely. The generalized $B$-series analyses the
  operators $\Fub{\ell}$ through their elementary
  differentials. Although beyond the scope of this study, if desired,
  such an approach can be used to refine the LEM calculation for
  a specific problem.
\end{rmk}

\subsection{Operator-splitting methods with optimal LEM}
\label{subsec:optimalLEM}

In order for a two-stage operator-splitting method applied to a
$2$-additive ODE~\cref{eq:2splitode} to achieve third-order accuracy,
the coefficients $\{\aaalpha{k}{\ell}\}_{k=1,2}^{\ell = 1,2}$ must
satisfy the following five equations:
\begin{equation}
	\label{OS23ordercond}
	\left\{
	\begin{aligned} 
		& \aaalpha{1}{1} + \aaalpha{2}{1}  = 1,  \\ 
		& \aaalpha{1}{2} + \aaalpha{2}{2}  = 1,  \\ 
		& \aaalpha{1}{2}\aaalpha{1}{1} + \aaalpha{2}{2}\left(\aaalpha{1}{1} + \aaalpha{2}{1} \right) = \frac{1}{2}, \\ 
		& \aaalpha{1}{2}\left(\aaalpha{2}{1} \right)^2  = \frac{1}{3}, \\ 
		& \aaalpha{1}{1}\left(\aaalpha{1}{2}  + \aaalpha{2}{2} \right)^2 + \aaalpha{2}{1} \left( \aaalpha{2}{2}  \right)^2  = \frac{1}{3}.
	\end{aligned} \right.
\end{equation}

The system \cref{OS23ordercond} has no solution. It is easy to show
that the set of coefficients
$\displaystyle \left\{\aaalpha{1}{1} = \frac{1}{3}, \aaalpha{2}{1} =
  \frac{2}{3}, \aaalpha{1}{2} = \frac{3}{4}, \aaalpha{2}{2} =
  \frac{1}{4} \right\}$ is the unique solution of the first four
equations. However, it fails to satisfy the fifth equation.  Hence, a
third-order, 2-split operator-splitting method \cref{eq:s-stageOS}
must have at least three stages. Therefore, we
start by considering $s=3$.


  Consider a three-stage, third-order, 2-split operator-splitting
  method with coefficients
  $\{\aaalpha{k}{\ell}\}_{k=1,2,3}^{\ell = 1,2}$. Solving the order
  conditions \cref{orders123}, we find the Ruth
  method~(\cref{tab:Ruthcoeff}) as an isolated solution, as well as
  two one-parameter families of solutions. Denoting the free parameter
  by $\OSTTparam$, the two families of solutions can be expressed as
\begin{align*}
	\aaalpha{1}{1} & = 1-\aaalpha{2}{1} - \aaalpha{3}{1}   , & \aaalpha{1}{2} & = 1- \aaalpha{2}{2} - \aaalpha{3}{2},  \\ 
	\aaalpha{2}{1} & = \frac{2\aaalpha{3}{1}\aaalpha{3}{2} - 2\aaalpha{3}{1} + 1}{2\aaalpha{1}{2}}   , & \aaalpha{2}{2} & = 
	\frac{1-\OSTTparam}{2} \pm \frac{\sqrt{144\OSTTparam^4 + 72\OSTTparam^3 - 99\OSTTparam^2 + 30\OSTTparam- 3}}{24\OSTTparam - 6}, \\
	\aaalpha{3}{1} & = -\frac{3\aaalpha{2}{2} + 3\aaalpha{3}{2} - 1}{6\aaalpha{2}{2}(\aaalpha{3}{2}-1)}   , & \aaalpha{3}{2} & = \OSTTparam. 
\end{align*}
For finite values to exist for
$\aaalpha{1}{1}, \aaalpha{3}{1}, \aaalpha{2}{2}$, and
$\aaalpha{3}{2}$, we must have
$\displaystyle \OSTTparam\neq \frac{1}{4}, \frac{1}{3}, 1$.  In
addition, to obtain real-valued solutions, we must have
$\OSTTparam > \displaystyle \frac{1}{4}$ or $\theta < -1.217077796$.
These restrictions on $\OSTTparam$ imply that none of the coefficients
vanish, and hence all three-stage, third-order, 2-split methods with
real coefficients must have six sub-integrations. We restrict our
attention to operator-splitting methods with real coefficients because
they are most widely used in practice.

Minimizing the LEM \cref{eq:lem_3rd} leads to the third-order
palindromic method of the embedded pair {\tt Emb 3/2 AKS}
\cite{auzinger2016practical}. We denote this method as AKS3. The
coefficients are given to 15 decimal places in
\cref{tab:Emb32AKScoeff}.
\begin{table}[htbp]
	\centering
	\caption{Coefficients $\alpha_k^{[\ell]}$ for the AKS3 method}
\begin{tabular}{|c|r|r|}
    \hline 
    \multicolumn{1}{|c|}{$k$} & \multicolumn{1}{c|}{$\displaystyle \alpha_k^{[1]}$} & \multicolumn{1}{c|}{$\displaystyle \alpha_k^{[2]}$}  \\
    \hline 
    1 & $0.268330095673069 $  & $0.919661524555154  $ \\
    \hline 
    2 & $-0.187991620228223 $  & $-0.187991620228223 $\\
    \hline 
    3 & $0.919661524555154  $  & $0.268330095673069 $ \\
    \hline 
\end{tabular}
	\label{tab:Emb32AKScoeff}
\end{table}


In~\cref{sec:beat_ruth_main}, we derive a method with optimized LEM by
adding an extra stage but with a minimal amount of increased
computation, i.e., only one extra sub-integration.  Then
in~\cref{sec:beat_ruth_num_exp}, we evaluate the performance of these
methods on the Niederer benchmark problem.

\subsection{Linear stability of the FSRK methods}
\label{subsec:linear_stab}

Some problems are stability-constrained, i.e., the size of the time
step used to solve a set of ODEs is governed by considerations of
stability and not accuracy. Accordingly, if an FSRK method is used to
solve a stability-constrained problem, it may be possible to improve
performance by enhancing its linear stability.

We note that because $z^{[1]}$ and $z^{[2]}$ in \cref{eq:fsrk_stab}
are independent variables, it is impossible to plot the traditional
linear stability function \cref{eq:fsrk_stab} in the complex
plane. However, in the case of reaction-diffusion problems, the
Jacobians of each operator with respect to the solution $\yub$ are
simultaneously diagonalizable in a neighbourhood of the
solution. Hence, we can scale $z^{[1]}$ or $z^{[2]}$ based on the
relative size of the eigenvalues of the two operators.  Henceforth, we
specialize our discussion to reaction-diffusion systems, noting that
similar analysis is applicable to other simultaneously diagonalizable
systems.

Let $\lambda^{[D]}$ and $\lambda^{[R]}$ be the most negative real
eigenvalues of the Jacobians of the diffusion operator $\Fub{D}$ and
reaction operator $\Fub{R}$, respectively. If the diffusion operator
is considered as operator $1$, then let
$\displaystyle z^{[1]} = \frac{\lambda^{[D]}}{\lambda^{[R]}}
z^{[2]}$. Denoting $z^{[2]}$ by $z$, the linear stability function
\cref{eq:fsrk_stab} can be written as the single-variable function
\begin{equation}
	\label{eq:fsrk_stab_single_var_dr}
	R_{DR}\left(z\right) = \prod\limits_{k=1}^s \stabfunc{k}{D} \left( \frac{\lambda^{[D]}}{\lambda^{[R]}} \aaalpha{k}{1} z \right) 
	\stabfunc{k}{R} \left(\aaalpha{k}{2} z\right),
\end{equation}
where $\stabfunc{k}{D}, \stabfunc{k}{R}$ are the stability functions of the
Runge--Kutta methods applied to the diffusion and reaction operators
at stage $k$, respectively.  If the reaction operator is considered as
operator $1$, then let
$\displaystyle z^{[2]} = \frac{\lambda^{[D]}}{\lambda^{[R]}}
z^{[1]}$. Denoting $z^{[1]}$ by $z$, the linear stability function
\cref{eq:fsrk_stab} can be written as the single-variable function
\begin{equation}
	\label{eq:fsrk_stab_single_var_rd}
	R_{RD}\left(z\right) = \prod\limits_{k=1}^s \stabfunc{k}{R} \left( \aaalpha{k}{1} z \right) 
	\stabfunc{k}{D} \left(\frac{\lambda^{[D]}}{\lambda^{[R]}} \aaalpha{k}{2} z\right), 
\end{equation}
where we note that in both
\cref{eq:fsrk_stab_single_var_dr,eq:fsrk_stab_single_var_rd}, the
variable $z = \lambda^{[R]} \Dt$. 
\begin{thm}
  Let $\{\aaalpha{k}{\ell}\}_{k=1,\dots,s}^{\ell=1,2}$ be the
  coefficients of an operator-splitting method $\Psi$ and let
  $\{\adjaaalpha{k}{\ell}\}_{k=1,\dots,s}^{\ell=1,2}$ be the
  coefficients of the adjoint method $\Psi^\ast$.
  Let $RK_k^{[D]}$ ($RK_k^{\ast [D]}$) and $RK_k^{[R]}$
  ($RK_k^{\ast [R]}$) be the Runge--Kutta methods used to solve the
  diffusion and reaction operator, respectively, at stage $k$ of the
  operator-splitting method $\Psi$ ($\Psi^\ast$). We further assume
  that $RK_k^{[D]} = RK_{s-k+1}^{\ast [D]}$ and
  $RK_k^{[R]]} = RK_{s-k+1}^{\ast [R]}$, for all $k=1,2,\dots, s$.  Then
  the stability function of solving a problem in the order of
  diffusion-reaction with the method $\Psi$ is the same as the
  stability function of solving the problem in the order of
  reaction-diffusion with the method $\Psi^\ast$.
  \label{thm:order_adj}
\end{thm}

\begin{proof}
	
  Because $\Psi$ and $\Psi^\ast$ are adjoints,
  $\aaalpha{k}{1} = \adjaaalpha{s-k+1}{2}$, and
  $\aaalpha{k}{2} = \adjaaalpha{s-k+1}{1}$. Therefore,
  $\{\aaalpha{k}{1}\}_{k=1,2,\dots,s} =
  \{\adjaaalpha{k}{2}\}_{k=1,2,\dots,s}$ and
  $\{\aaalpha{k}{2}\}_{k=1,2,\dots,s} =
  \{\adjaaalpha{k}{1}\}_{k=1,2,\dots,s}$, and hence the stability
  function \cref{eq:fsrk_stab_single_var_dr} of solving a problem in the order of
  diffusion-reaction using the method $\Psi$ is identical to the
  stability function \cref{eq:fsrk_stab_single_var_rd} of solving a problem in
  the order of reaction-diffusion using the adjoint method
  $\Psi^\ast$.

\end{proof}

Strictly speaking, analysis based on~\cref{eq:fsrk_stab_single_var_dr} is
for a \emph{scalar} ODE. In such a case, it is useful to think of
eigenvalues $\lambda^{[D]}$ and $\lambda^{[R]}$ as being given (from
the problem) and then the stability of a given method can be analyzed
as the stepsize $\Dt$ (and hence $z$) is varied. However, this
analysis cannot be directly applied to the \emph{system} of ODEs that
arises from the method of lines. For these problems, a given spatial
discretization leads to a distribution of eigenvalues that will all be
present when solving the ODEs. Accordingly, standard plots of
stability regions need to be interpreted more in the sense of
extremes, as described below.

The most celebrated way to solve a 2-split ODE~\cref{eq:2splitode} is
the IMEX approach, i.e., to treat one operator implicitly and the
other explicitly.  \Cref{fig:explain_relevant_region} gives two
examples of stability regions for FSRK methods that use one implicit
and one explicit sub-integrator. \cref{fig:explain_relevant_region1}
depicts a stability region that consists of multiple disjointed
parts. In this case, only the part of the stability region that
contains the origin (shaded) is relevant to the stability of the
method. Then for example, assuming all eigenvalues are real and
negative, the largest stable stepsize can be estimated from
$\lambda\, \Dt \approx -5.8$, where $\lambda$ is the most negative
eigenvalue of the operators. The other part (not shaded) is not
relevant in practice for method-of-lines ODEs.  A stable stepsize
cannot be determined from $\lambda\, \Dt \approx -8.8$ because there
will generally be $z$ values (combinations of $\lambda\, \Dt$ for
given $\Dt$) that land in the unstable region between the two regions
and hence lead to instability of the overall method.

\cref{fig:explain_relevant_region2} depicts a stability region that
has a \emph{hole of instability}.  The potential existence of such
holes can be understood as follows. We recall that the linear
stability function 
of an implicit Runge--Kutta method is a rational function and hence
has poles. For normal (forward-in-time) integration, these poles are
usually in the right-hand complex plane. However, when integrating
backwards in time, these poles are now in the left-hand complex plane
and manifest as holes of instability when they occur inside a
stability region. Accordingly, in the presence of dominant negative
real eigenvalues, a method's practical stability is limited by the
right-most negative $x$-intercept of $|R(z)| = 1$ in the left-hand
complex plane.  More formally, the practical stability region of an
FSRK method applied to a method-of-lines system of ODEs is the
intersection of $|R(z)| \le 1$ that contains the origin and
$\{ z = x+iy ~|~ \xmaxintercept \leq x \leq 0 \} $, where
$\xmaxintercept$ is the right-most negative
$x$-intercept of
$|R(z)| = 1$.

\begin{figure}[!htbp]
	\begin{subfigure}{\textwidth}
		\centering
		\includegraphics[width=\textwidth]{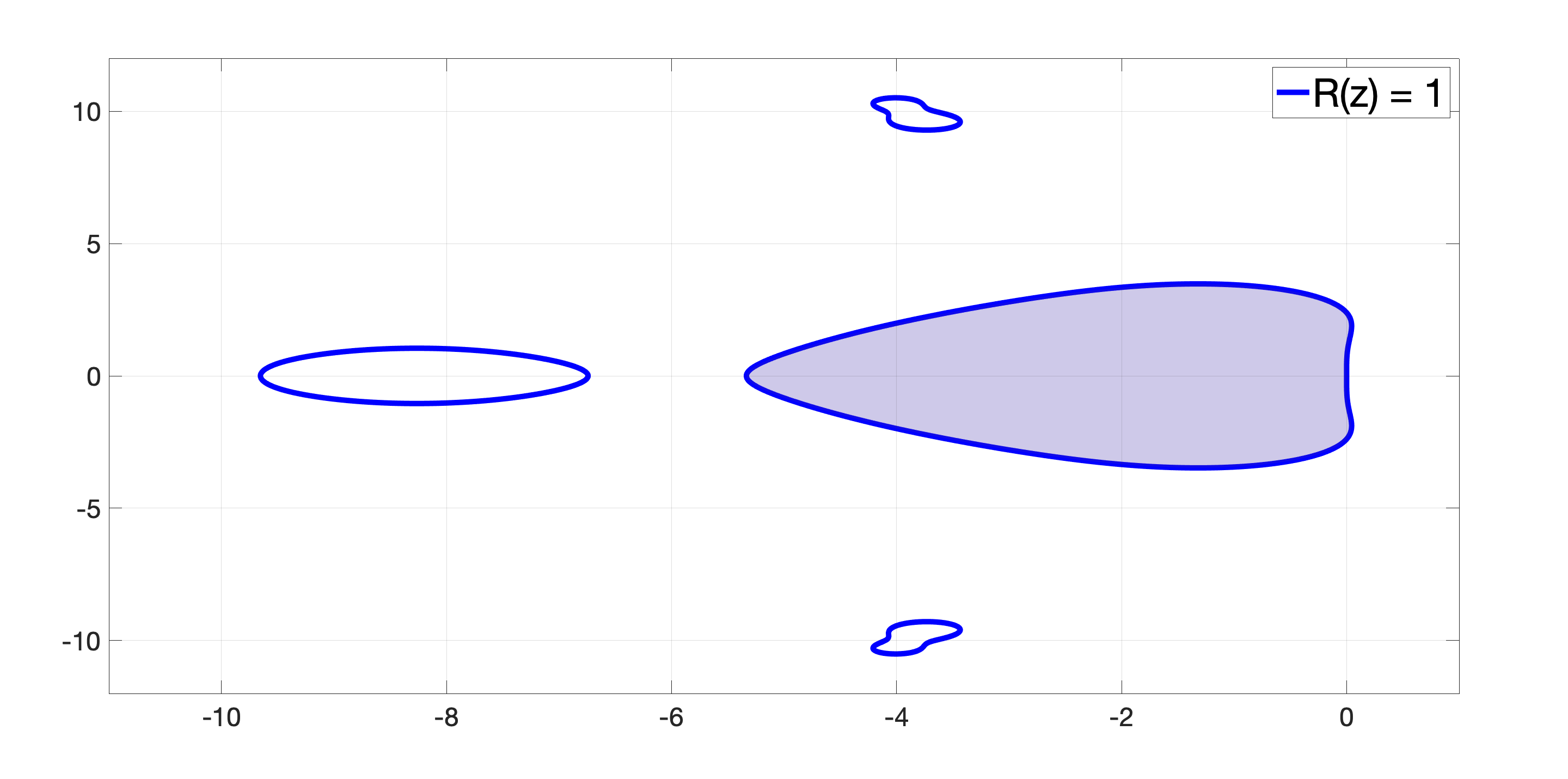}	
		\caption{}
		\label{fig:explain_relevant_region1}
	\end{subfigure}%
	
	\begin{subfigure}{\textwidth}
		\centering
		\includegraphics[width=\textwidth]{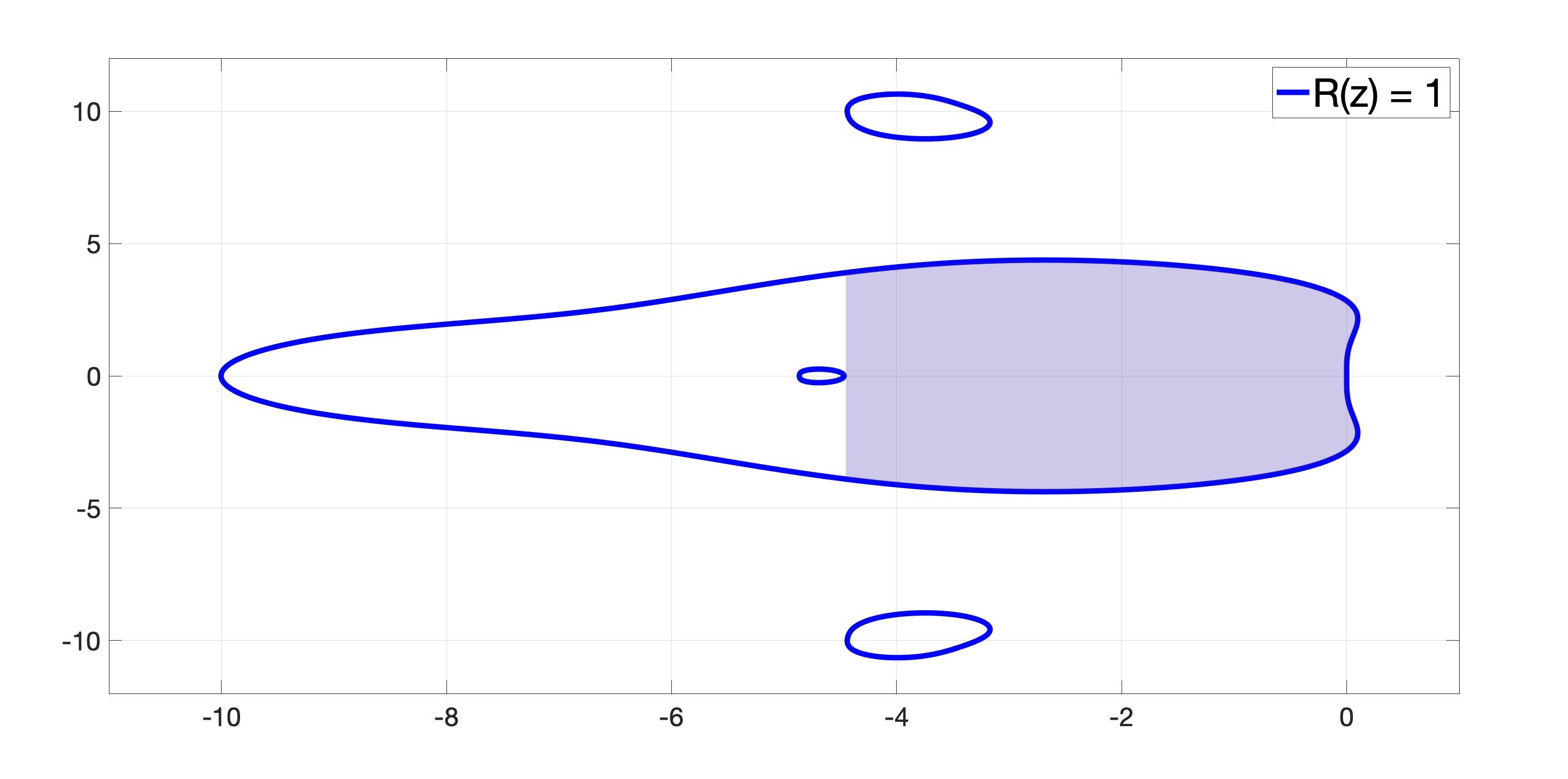}	
		\caption{}
		\label{fig:explain_relevant_region2}
	\end{subfigure}
	\caption{Examples of stability regions of FRSK methods that use
      one implicit and one explicit sub-integrator. Sub-regions
      relevant to stability of the method in practice when solving
      method-of-lines ODEs are shaded.
      \cref{fig:explain_relevant_region1}: When a stability region
      consists of multiple disjointed parts, only the part that
      contains the origin is relevant to the stability of the method
      in practice. \cref{fig:explain_relevant_region2}: If a hole of
      instability is present in the part of the stability region
      containing the origin, the right-most $x$-co-ordinate of the
      hole determines the largest stable step size in the presence of
      dominant negative real eigenvalues. }
	\label{fig:explain_relevant_region}
\end{figure}

Theoretically, one can choose the sub-integration methods with full
flexibility to optimize the stability region. For example, we often
solve the stiff operator with an implicit method and the non-stiff
operator with an explicit method. However, stability is not the only
consideration when choosing a sub-integration method in practice;
e.g., the computational cost of the non-linear solve may tip the
balance in favor of an explicit sub-integration method. Therefore, for
a given problem, we may choose the sub-integration method based on
computational instead of stability considerations. We then focus on
the choice of operator-splitting coefficients to improve the stability
of the overall FSRK method.

Based on this analysis, we propose two strategies to
mitigate the instability due to backward-in-time integration and at
the same time reduce computational cost per step:
\begin{enumerate}
\item Choose operator-splitting coefficients and operator ordering so
  that the right-most negative $x$-intercept of $|R(z)| = 1$ is
  located far from the origin. As noted by \cref{thm:order_adj},
  changing the operator ordering is equivalent to solving the original
  operator ordering problem with the adjoint method.
  
\item Use an explicit Runge--Kutta method for backward-in-time
  sub-integrations. In this scenario, the idea is that an explicit
  Runge--Kutta method improves stability by removing the pole
  completely (because its stability function is a polynomial), and it
  also improves efficiency because explicit methods are generally less
  computationally expensive than implicit ones.
\end{enumerate}
The effects of these strategies are demonstrated
in~\cref{sec:beat_ruth_num_exp}.

 
\section{Main results}
\label{sec:beat_ruth_main}

In his 1968 paper~\cite{Strang1968}, Strang mentioned three
fundamental criteria when comparing numerical methods: accuracy,
simplicity, and stability. In addition to the order of accuracy, we
can characterize the accuracy of a method for a given order based on
the local error measure \cite{auzinger2016practical}.  Simplicity can
be interpreted as the complexity of a method, or in this case, the
number of sub-integrations required. Thus, it is associated with the
computational cost of a method. The linear stability function can be
used to characterize the stability of a fractional-step Runge--Kutta
method~\cite{spiteri_wei_FSRK}. In this section, we derive third-order
operator-splitting methods by optimizing the linear stability region
and local error measure. We include the simplicity (in the sense of
computational cost) of the method in our analysis and discuss the
effect of the operator ordering on the stability of the
method. 
Accordingly, we minimize the LEM or $\hat{x}$ for a given number of stages
while minimizing as well the number of sub-integrations. We denote an
$\Nop$-split operator-splitting method with $s$ stages, order $p$, and
$f$ sub-integrations by \osNspf{\Nop}{s}{p}{f}.

\subsection{\osNspf{2}{4}{3}{7} method with optimized LEM}
\label{subsec:LEM_os}
To improve upon the LEM of the methods presented
in~\cref{subsec:optimalLEM} with minimal additional computational
cost, we consider four-stage, third-order, 2-split methods with 7
sub-integrations; i.e., one of the coefficients
$\{\aaalpha{k}{\ell}\}_{k=1,2,3,4}^{\ell=1,2}$ vanishes. Determination
of such methods is performed as follows. For each
$\aaalpha{k}{\ell}=0$, $k=1,2,3,4 $, $\ell=1,2$, we minimize the LEM
\cref{eq:lem_3rd} using {\tt MATLAB}'s {\tt GlobalSearch} algorithm
starting from 50 random seeds. From the candidate methods generated,
we found two methods with the minimal LEM 
that are also adjoints of each other. We refer to one of these optimum
methods as \osftsminLEM. We note that there is no loss of generality
through the choice made at this point because we have yet to specify
operator ordering. The coefficients of \osftsminLEM\ are given to 15
decimal places in \cref{tab:osftscoeff}.
\begin{table}[htbp]
	\centering
	\caption{Coefficients $\alpha_k^{[\ell]}$ of the \osftsminLEM~method}
\begin{tabular}{|c|r|r|}
    \hline 
    \multicolumn{1}{|c|}{$k$} & \multicolumn{1}{c|}{$\displaystyle \alpha_k^{[1]}$} & \multicolumn{1}{c|}{$\displaystyle \alpha_k^{[2]}$}  \\
		\hline 
		1 & $0.675603619637542 $ & $  1.351207213243766$ \\
		\hline 
		2 & $-0.175603577692365 $ & $  -1.702414383919316$\\
		\hline 
		3 & $-0.175603614267295$ & $    1.351207170675550$ \\
		\hline 
		4 & $0.675603572322118  $ &  \multicolumn{1}{c|}{$0$}  \\
		\hline 
	\end{tabular}
	\label{tab:osftscoeff}
\end{table}


The required number of sub-integrations and the LEM (rounded to two
decimal places) for Ruth, AKS3, and \osftsminLEM~are given in
\cref{tab:comp_method}. Although the \osftsminLEM~method requires one
more sub-integration than Ruth and AKS3, it may have better general
accuracy properties given its comparatively small LEM.  Hence, in
principle, it may be able to overcome the added expense of an extra
sub-integration per step by being able to take larger steps while
maintaining the same accuracy and achieve improved
performance. However, it turns out that the Niederer benchmark problem
considered in~\cref{sec:beat_ruth_num_exp} is too
stability-constrained for this to be the case; see
also~\cref{fig:osfts_DR_RD} below.
\begin{table}[htbp]
	\centering
	\caption{Summary of the properties of Ruth, AKS3, and \osftsminLEM. }
	\begin{tabular}{|c|c|c|}
		\hline 
		Method & $\#$ of sub-integrations &  LEM  \\
		\hline 
		Ruth & $6$ & $0.36$ \\
		\hline 
		AKS3  & $6$ & $0.25$ \\ 
		\hline 
		\osftsminLEM  & $7$ & $6.55 \num{1e-8}$ \\
		\hline 
	\end{tabular}
	\label{tab:comp_method}
\end{table}


\subsection{Third-order, 2-split methods with optimized
  linear stability}
\label{subsec:linear_stab_os}

As is shown in \cref{subsec:linear_stab}, given the most negative real
eigenvalues $\lambda^{[D]}$ and $\lambda^{[R]}$ of each operator and
the Runge--Kutta sub-integrators for each operator at each stage, the
stability functions \cref{eq:fsrk_stab_single_var_dr} or \cref{eq:fsrk_stab_single_var_rd} of an FSRK method
depends the operator-splitting coefficients
$\displaystyle \{ \alpha_k^{[\ell]}\}_{k=1,2,\dots,s}^{\ell=1,2}$. 
Fixing an operator ordering, we can achieve better stability,
especially for problems with strictly real eigenvalues, by designing a
method that minimizes $\xmaxintercept$, the right-most negative
$x$-intercept of $|R(z)| =1$.  When the Jacobians of the operators are
simultaneously diagonalizable, this value is a good predictor of the
largest stable step size in practice.
We note that the adjoint of the optimal method can be used to solve
the problem with the reversed operator ordering as indicated by
\cref{thm:order_adj}.

Because we are interested in third-order operator-splitting methods
with one implicit and one explicit sub-integrator, we use the
two-stage, third-order SDIRK method with
$\gamma = \displaystyle \frac{3+\sqrt{3}}{6}$ (SDIRK(2,3))
\cite{HairerNorsettWanner1993} and Kutta's third-order explicit method
(RK3) \cite{kutta1901}.
%
%
The stability functions of SDIRK(2,3) and RK3 are given by
\begin{align*}
	\StabFunc{\text{SDIRK(2,3)}}(z) & =  1 - \frac{z^2(2\gamma-1)}{2(\gamma z-1)^2} - \frac{z}{\gamma z-1}, \quad \gamma =  \frac{3+\sqrt{3}}{6}, \\
	\StabFunc{\text{RK3}}(z) & =  1 +z + \frac{z^2}{2} + \frac{z^3}{6}. 
\end{align*}

The quantity $\xmaxintercept$ is
minimized using the {\tt MATLAB} {\tt GlobalSearch} algorithm with
$100$ random initial guesses for each case. 
We restrict the design
space of each operator-splitting coefficient to the interval $[-1,1]$
because excessively large sub-steps are generally
undesirable. Furthermore, backward sub-steps may be unstable, and
accordingly, their lengths should generally be minimized.  The
operator-splitting method that gives the most negative
$\xmaxintercept$ among the all candidates is then selected as the
optimal method for the particular operator ordering.


In~\cref{sec:beat_ruth_num_exp}, we use the Niederer benchmark problem
to explicitly demonstrate the performance of \osNspf{2}{3}{3}{6}\ and
\osfts\ methods with designed by optimizing linear stability and then
further enhanced through judicious sub-integration.

 
\section{Numerical experiments}
\label{sec:beat_ruth_num_exp}

The monodomain model is a popular mathematical model for cardiac
electrophysiology~\cite{sundnesbook}.  It consists of a partial
differential equation (PDE) that models the electrical activity in
(co-located) intracellular and extracellular domains of myocardial
tissue, coupled with a system of non-linear ODEs that describe the
current density flowing through the membrane ionic channels. The
monodomain model is a simplification of the bidomain model under the
assumption that the intracellular conductivity is in constant
proportion to the extracellular conductivity. Nonetheless, the
monodomain model has practical utility along with its computational
advantages~\cite{sundnesbook}. Both the monodomain and bidomain models
are computationally expensive because of the high resolution required
from both spatial and temporal discretizations and the increasing size
and complexity of the ionic membrane models. For these reasons,
operator-splitting methods are used to achieve better efficiency and
feasibility over monolithic approaches~\cite{sundnesbook}. Software
libraries for cardiac simulation typically implement the first-order
Lie--Trotter method or the second-order Strang method for the
simulation of the monodomain/bidomain
models~\cite{sundnesbook,Plank2021_openCARP,Cooper2015_Chaste}.
High-order methods have generally been considered unstable due to the
backward-in-time integration. However, recent work has shown this is
not necessarily the case~\cite{spiteri_wei_FSRK}, in particular for
the monodomain and bidomain models~\cite{cervi2018, cervi2018accuracy,
  cervi2019efficiency}. In this section, we demonstrate various ways
to improve the stability and efficiency of high-order
operator-splitting methods to solve a well-known benchmark problem for
the monodomain model.

A benchmark problem was studied in Niederer et
al.~\cite{niederer2011} to verify and compare the performance of
solutions of eleven myocardial tissue electrophysiology simulators. The
benchmark problem consists of a 3D monodomain model described by the
differential equations,
\begin{equation}
	\label{niederer}
	\begin{aligned}
      \chi C_m \pdv{v}{t} + \chi \Iion(\bs,v) &= \nabla \cdot (\sigma \nabla v), \\
      \pdv{\bs}{t} &= \ff(t,\bs,v),
	\end{aligned}
\end{equation}
where $v$ is the transmembrane potential (or \emph{voltage}), $\bs$ is
a set of variables describing the state of the cell, $\chi$ is the
cell surface-to-volume ratio, $C_m$ is the specific membrane
capacitance, and $\sigma$ is the conductivity tensor.  The
transmembrane current density, $\Iion$, is defined by the cell model,
whose states are governed by the ODEs with right-hand side $\ff(t,\bs,v)$.

The Niederer benchmark problem is posed on a cuboid of dimension
$0.3$~cm $\times\ 0.7$~cm $\times\ 2$~cm and time interval $[0,40]$~ms
with a specified initial stimulus at a corner of the domain. The cell
model used is that of ten Tusscher and Panfilov \cite{tt2006}, a cell
model of human epicardial myocytes. This cell model consists of 19
ODEs that model all the major ion channels and intracellular calcium
dynamics. 
Details of the remaining parameters of the problem are given in
\cite{niederer2011}.  To solve the Niederer benchmark problem
numerically, we first discretize the 3D domain using central finite
differences with a spatial mesh of size $\Dx = \Dy = \Dz = 0.05$~cm.
A reference solution was computed by solving the problem with Strang
splitting using the {\tt RK45} method in {\tt
  scipy.integrate.solve\_ivp} with {\tt rtol=1e-3} and {\tt atol=1e-6}
as sub-integrators, and halving the time step until 3 digits matched
across $4305$ equally spaced spatial points and $21$ temporal points
$[0:2:40]$~ms.

To solve the Niederer benchmark problem using operator splitting, we
split the resulting reaction-diffusion system into two operators, 
representing reaction and diffusion, respectively: 
\begin{equation}
	\Fub{R} = \left\{  
	\begin{aligned}
	\pdv{v}{t} & = -\frac{1}{C_m}\Iion(\bs,v), \\
	\dv{\bs}{t} & = \ff(t,\bs,v),
	\end{aligned}
	\right. 
\text{ and } \quad
	\Fub{D} =  \left\{  
	\begin{aligned}
	\pdv{v}{t} & = \frac{1}{\chi C_m} \nabla \cdot (\sigma \nabla v), \\
	\dv{\bs}{t} & = \mathbf{0}.
	\end{aligned}
  \right.
\end{equation}
We solve the Niederer benchmark problem with different
operator-splitting methods and different Runge--Kutta methods as
sub-integrators. To measure the error of an FSRK method, we use the
mixed root-mean-square (MRMS) error~\cite{marsh2012} of quantity $X$
[MRMS]$_X$ at $M = 4305 \times 21$ space-time points defined by
\begin{equation*}
  [\text{MRMS}]_X = \sqrt{\frac{1}{M}
    \sum\limits_{i=1}^M \left(\frac{X_i - X_i^{\text{ref}}}{1+\abs{X_i^{\text{ref}}}} \right)^2},
\end{equation*}
where $X_i$ and $X_i^{\text{ref}}$, respectively, denote the numerical
solution and the reference solution.  Because the transmembrane
potential $v$ is the main variable of interest, we focus on [MRMS]$_v$
for the analysis. In practice, we set a level $[\text{MRMS}]_v=0.05$
as an acceptable error threshold. To examine the efficiency of these
methods, we find the largest constant step sizes $\Dt$ to two
significant figures that yields [MRMS]$_v \le 0.05$. These step sizes
are then used to solve the Niederer benchmark problem over the
interval $[0,40]$.
The simulations were performed using the \pythOS\ operator splitting
library \cite{spiteri_wei_guenter_pythos} on an HP Z820 workstation
running 64-bit Ubuntu 20.04 LTS with Linux kernel 5.4.0-177-generic.
The machine has an Intel Xeon E5-2650 2.0 GHz octa-core CPU and 512 GB
of 1600 MHz DDR3 RAM.  CPU times using are reported as the minimum of
three runs after checking for consistency.

For the Niederer benchmark problem, the most negative eigenvalues of
the Jacobians of the diffusion and reaction operators are
$\lambda^{[D]} \approx -1.92$ 
and $\lambda^{[R]} = -1260$ \cite{spiteri2010}, respectively.  For a
given FSRK method, we form stability functions $R_{DR} (z)$, where the
first operator is the diffusion operation and the second operator is
the reaction operator, and $R_{RD} (z)$, where the first operator is
the reaction operation and the second operator is the diffusion
operator, as given in \cref{eq:fsrk_stab_single_var_dr}.

\subsection{Optimized LEM}

The stability regions of \osftsminLEM\ for the DR and RD orderings are
shown in
\cref{fig:osfts_DR_RD}. Comparing~\cref{fig:osfts_DR_RD,fig:DR_RD_comp},
we see that the stability regions for both DR and RD orderings of
\osftsminLEM\ are comparable to AKS3 but significantly smaller than
those with optimized $\hat{x}$. These results suggest that
\osftsminLEM\ would underperform \osftsDRminx\
derived in \cref{subsect:opt_stab}.  Indeed, numerical experiments
confirm that optimizing LEM does not lead to better performance for
this example. Hence, we do not report on the \osftsminLEM\ method
further.

	\begin{figure}[htbp]
		\centering
		\includegraphics[width=\textwidth]{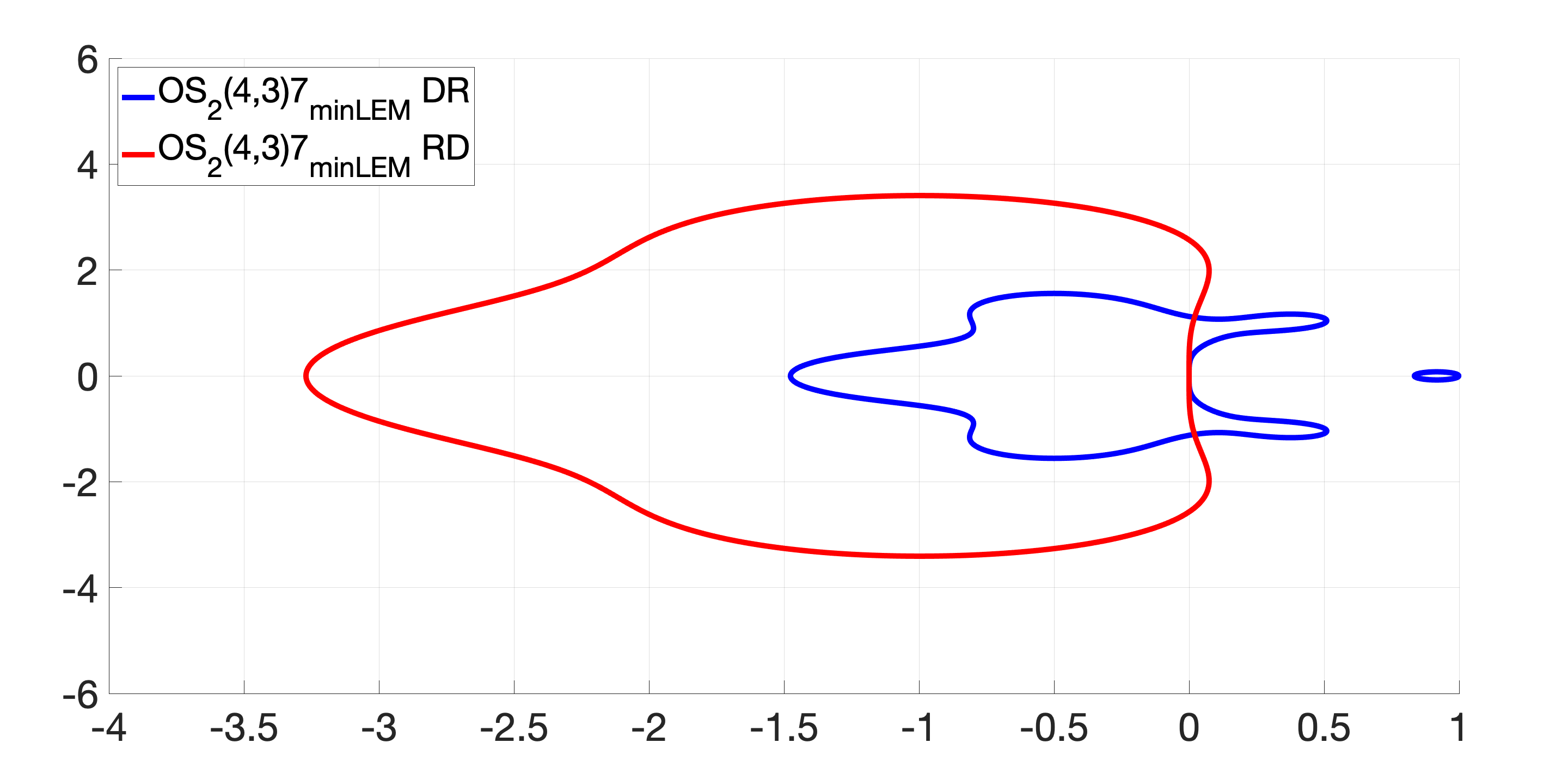}
		\caption{Stability regions of the \osftsminLEM\ method
			applied to the Niederer benchmark problem with different
			operator orderings.}
		\label{fig:osfts_DR_RD}
	\end{figure}

\subsection{Optimized linear stability}
\label{subsect:opt_stab}
For the case of \osNspf{2}{3}{3}{6}\ methods, we find the best result
is a method whose $\xmaxintercept$ is marginally more negative than
that of the Ruth method, which implies that it will allow a slightly
larger step-size than that allowed by the Ruth method. However,
because the computational cost per step for this method is the same as
the Ruth method, it is not expected this method will achieve a
significant improvement in the overall efficiency. Hence, the class of
\osNspf{2}{3}{3}{6}\ methods is not considered further.
		
For the class of \osfts\ methods, the
coefficients of the best method in the DR ordering are given to 15
decimal places in \cref{tab:os437DR_minx}. We denote this method by
\osftsDRminx. The best method with the RD ordering is the adjoint of \osftsDRminx. 

	\begin{table}[htbp]
		\centering
		\caption{Coefficients $\alpha_k^{[\ell]}$ for the \osftsDRminx~method.}
		\begin{tabular}{|c|r|r|}
			\hline 
			\multicolumn{1}{|c|}{$k$} & \multicolumn{1}{c|}{$\displaystyle \alpha_k^{[1]}$} & \multicolumn{1}{c|}{$\displaystyle \alpha_k^{[2]}$}  \\
			\hline 
			1 & \multicolumn{1}{c|}{$0$} & $ 0.214870149852186$ \\
			\hline 
			2 & $0.511486052225367$  & $0.668690687888393$\\
			\hline 
			3 & $-0.501427388979812$  & $-0.041956908041494$ \\
			\hline 
			4 & $0.989941336754445$  & $0.158396070300915$ \\
			\hline 
		\end{tabular}
		\label{tab:os437DR_minx}
	\end{table}

%

\subsection{Operator ordering for optimal stability}

%

As discussed in \cref{subsec:linear_stab}, the order of the operators
should be chosen that the right-most negative $x$-intercept of
$|R(z)| =1$ is far from the origin. In \cref{fig:DR_RD_comp}, we
compare the stability regions of the Ruth, AKS3, and \osftsDRminx\
methods applied to the Niederer benchmark problem with different
operator orderings: diffusion-reaction (DR) and reaction-diffusion
(RD). The interior regions correspond to $|R(z)| < 1$. The stability
plots show that when using the Ruth method, the RD ordering is more
stable, and when using the \osftsDRminx\ method, the DR ordering is
more stable. These orderings suggest that it is preferable to
integrate the stiff reaction operator backward for as short an
interval as possible.

\begin{figure}[htbp]
	\centering
	\includegraphics[width=0.8\textwidth]{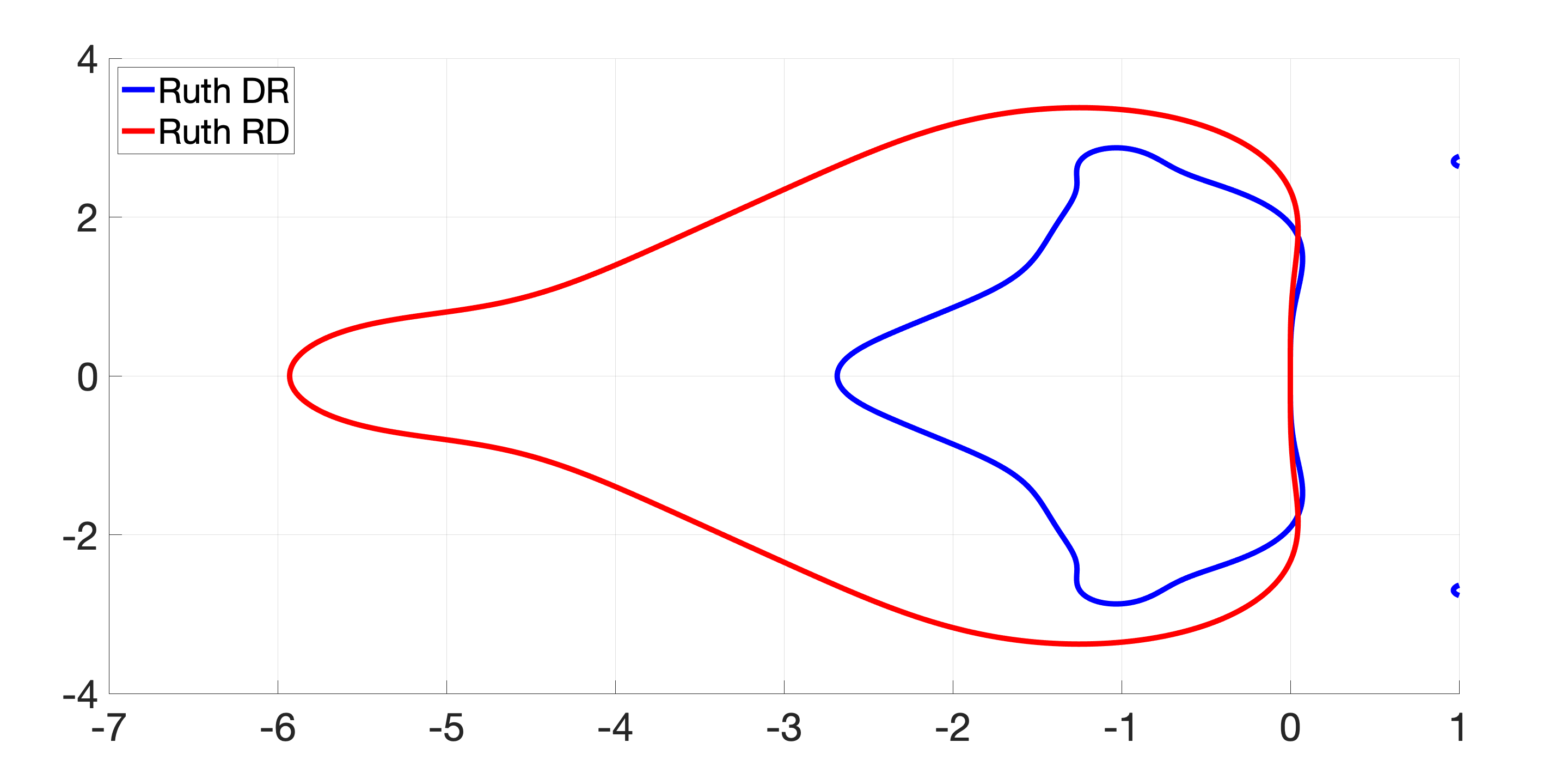}
	\includegraphics[width=0.8\textwidth]{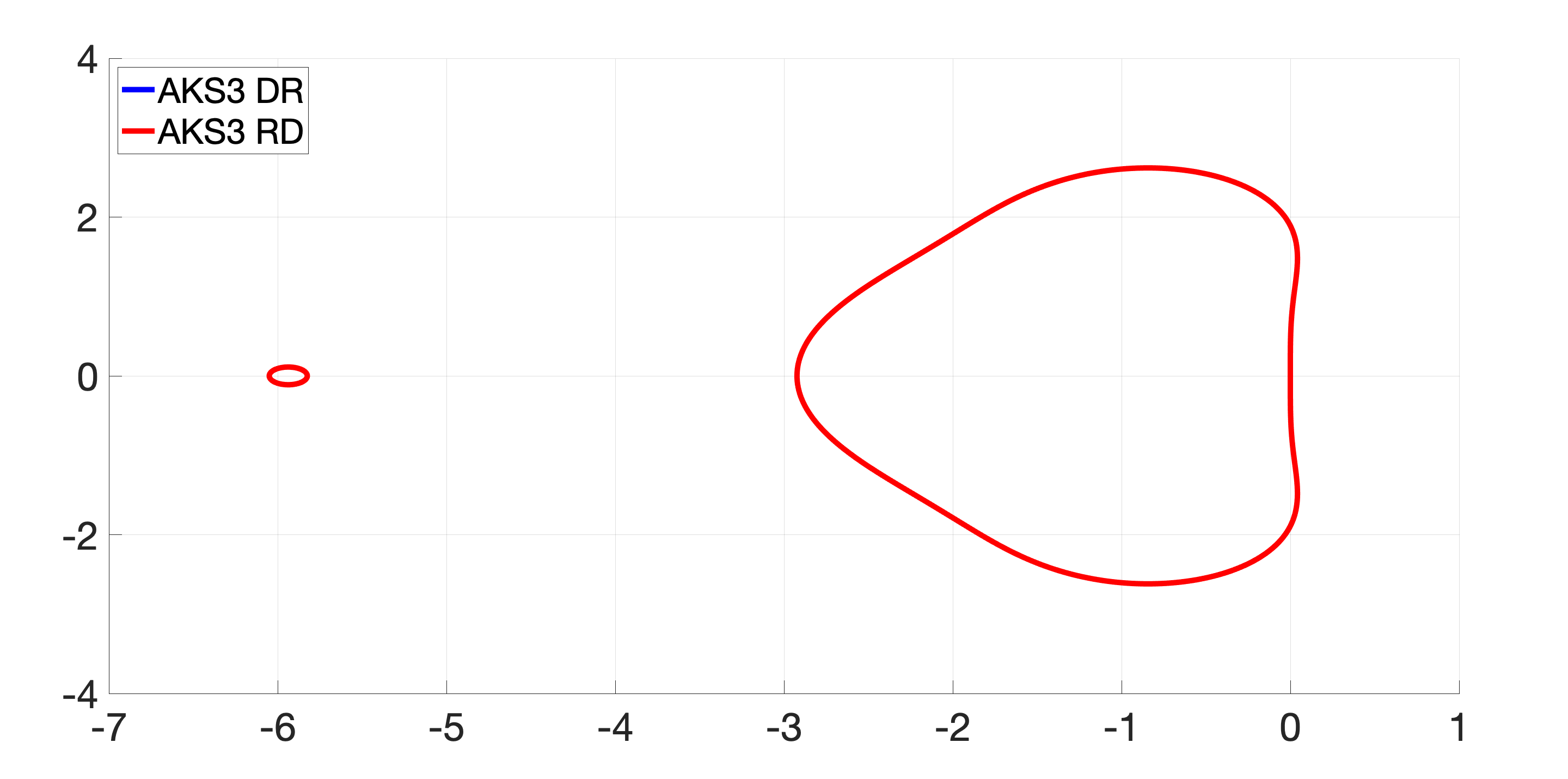}
	\includegraphics[width=0.8\textwidth]{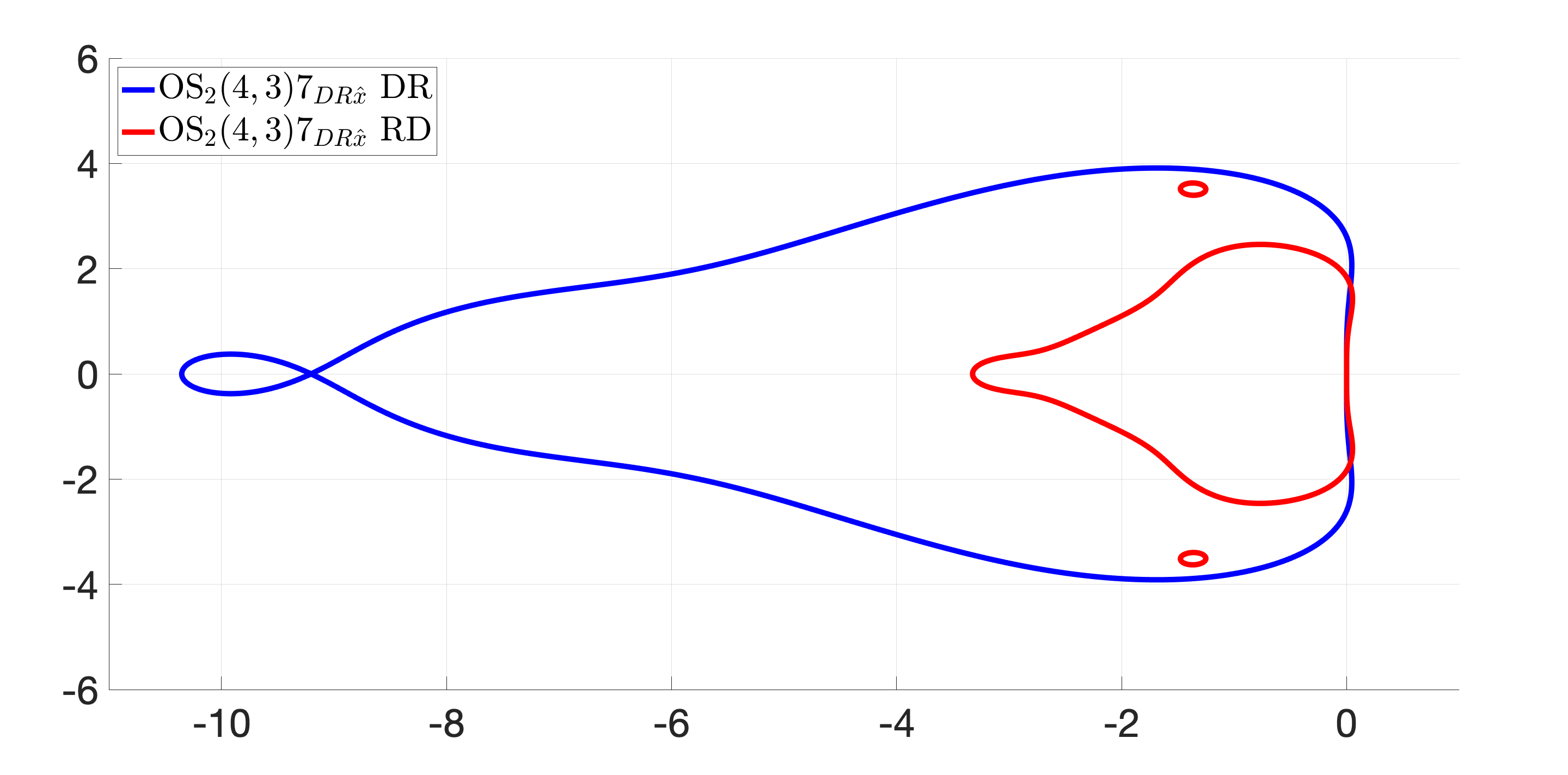}
	\caption{Stability regions of the Ruth, AKS3, and \osftsDRminx\ methods applied to the Niederer benchmark problem
      with different operator orderings.}
  	\label{fig:DR_RD_comp}
\end{figure}

We note that, in this case, no hole of instability appears in the main
stability region; therefore, the stability of the methods is not
dictated by the location of the pole on the left-hand side of the
complex plane. If the stability region is restricted by a hole as in
\cref{fig:explain_relevant_region2}, the operator ordering that places
the pole further to the left normally improves the stability
property~\cite{spiteri_wei_FSRK}.

\subsection{Performance of \osNspf{2}{4}{3}{7} methods}

The CPU times using the Ruth, AKS3, and \osftsDRminx\
methods are reported in \cref{tab:DR_RD_comp}.  The results confirm
the optimal integration orderings suggested in \cref{fig:DR_RD_comp},
i.e., the Ruth method can take a larger stable step size with operator
ordering RD and \osftsDRminx\ can take a larger stable step size with
operator ordering DR. We note that this computation is
stability constrained because the [MRMS]$_v$ errors for the allowed
step sizes are well below $0.05$. The \osftsDRminx\ 
method exhibit a $36\%$ efficiency gain compared to the Ruth method
(in RD order). AKS3 is less stable than the Ruth method for this
problem, and so the advantage of a smaller LEM does not lead to a
practical advantage for this (stability-constrained) problem.

\begin{rmk}
  \cref{fig:DR_RD_comp} suggests that the
  \osftsDRminx\ method applied in the RD ordering, i.e., in the opposite ordering
  for which it was designed, has a much smaller stable step size
  compared to the Ruth method. In other words, this new method is
  inefficient when not used according to its design
  principles. Accordingly, we omit details of this combination in
  \cref{tab:DR_RD_comp}.
	
\end{rmk}

\begin{table}[htbp]
	\centering
	\begin{tabular}{|c|c|c|c|c|c|c|}
		\hline 
		method & \multicolumn{3}{c|}{Diffusion-Reaction} & \multicolumn{3}{c|}{Reaction-Diffusion} \\ 
		\hline 
		& $\Delta t$ & time & MRMS error & $\Delta t$ & time & MRMS error  \\ 
		\hline 
		Ruth  &   0.0028 &	8831 &	0.00067 & 0.0062	& 3962 &	0.00039 \\ 
		\hline 
		AKS3 & 0.0031 &	7654 &	0.023 &	0.0031 &	7969 &	0.022 \\ 
		\hline 
      \osftsDRminx & \textbf{0.011} & \textbf{2509} & \textbf{0.00055} & --- & --- & ---  \\
      \hline 
	\end{tabular}
	\caption{Efficiency comparison of solving the Niederer benchmark
      problem using the Ruth, AKS3, and \osftsDRminx\
      methods. }.
   \label{tab:DR_RD_comp}
\end{table}

\subsection{Use of explicit Runge--Kutta methods for backward-in-time
  integration}

A further way to improve the linear stability of a fractional-step
Runge--Kutta method is to remove any poles in the linear stability
region in the left-hand side of the complex plane. Such poles can be
caused by backward-in-time sub-integration with an implicit
Runge--Kutta method. In such cases, the pole can be removed by
replacing an implicit Runge--Kutta sub-integrator with an explicit
Runge--Kutta sub-integrator when integrating backward in time. Because
the Niederer benchmark problem is stability constrained and not
accuracy constrained for the parameters chosen, we replace the
negative steps in both operators with the forward Euler (FE)
method. Interestingly, when integrating backward in time, the FE
method in fact leads to a method with a larger linear stability region
than Heun's method, which is in turn more stable than Kutta's
third-order method, because, for $z$ values with positive real parts,
$|\StabFunc{\text{FE}}(z)| < |\StabFunc{\text{Heun}}(z)| <
|\StabFunc{\text{RK3}}(z)|$.  Moreover, although use of the FE method
negatively impacts the order of the overall method (and generally the
accuracy of a given numerical solution), it has the least
computational expense per step, and so the increase in error does not
negate the increase in efficiency in this situation.

  As shown in \cref{fig:DR_RD_FE_comp1,fig:DR_RD_FE_comp2}, replacing
  both SDIRK(2,3) and RK3 with FE for the negative steps made slight
  improvement in stability of all the Ruth, AKS3, \osftsDRminx\ methods.
  

  As shown in \cref{tab:DR_RD_FE_comp}, although the use of the
  forward Euler method increased the error in each computation, all
  errors are still acceptable. All methods, moreover, demonstrate
  significant efficiency gains. Although the improvement in stability
  leads to no further improvement in step size in this case, the
  \osftsDRminx\ method is approximately $25\%$
  faster than the classical implementation reported in
  \cref{tab:DR_RD_comp} and $29\%$ faster than the optimal
  implementation of the Ruth method.

\begin{figure}[!htbp]
	\centering
	\includegraphics[width=\textwidth]{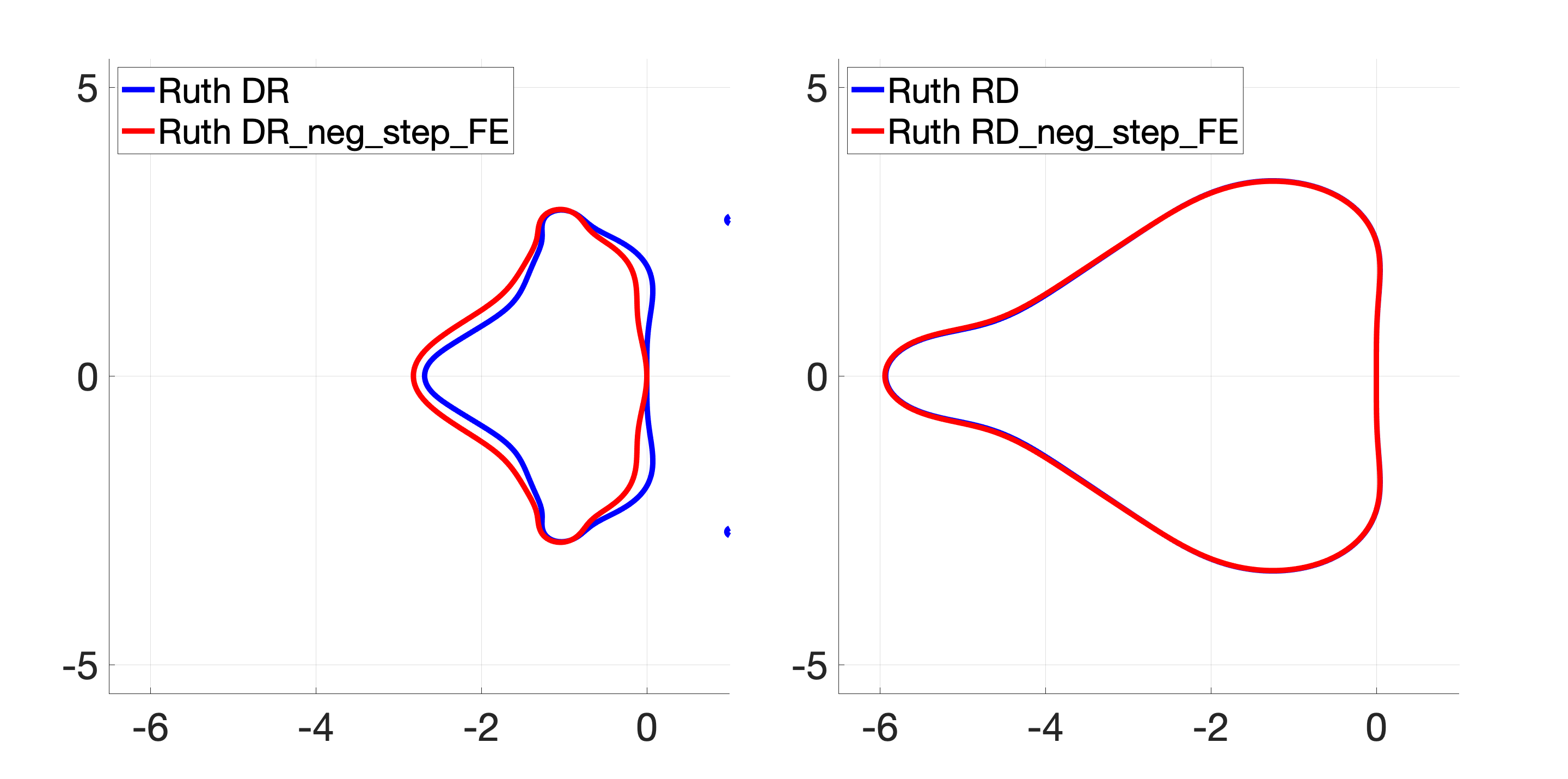}
	\includegraphics[width=\textwidth]{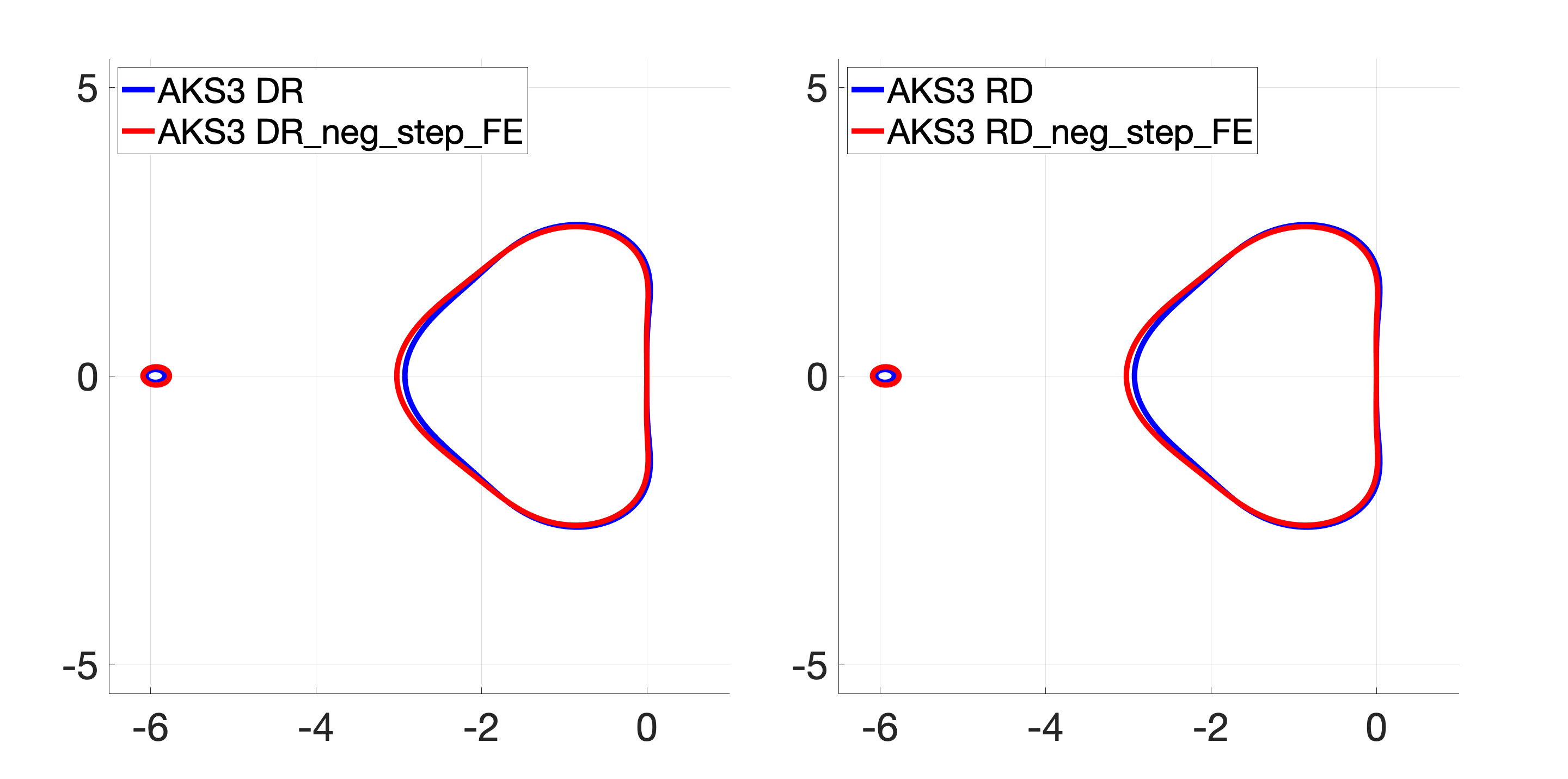}
		\caption{Stability regions of the Ruth and AKS3 applied to
		the Niederer benchmark problem with different operator orderings
		and FE applied to both negative steps.}
	\label{fig:DR_RD_FE_comp1}	
\end{figure}

\begin{figure}[htbp]
	\centering
	\includegraphics[width=\textwidth]{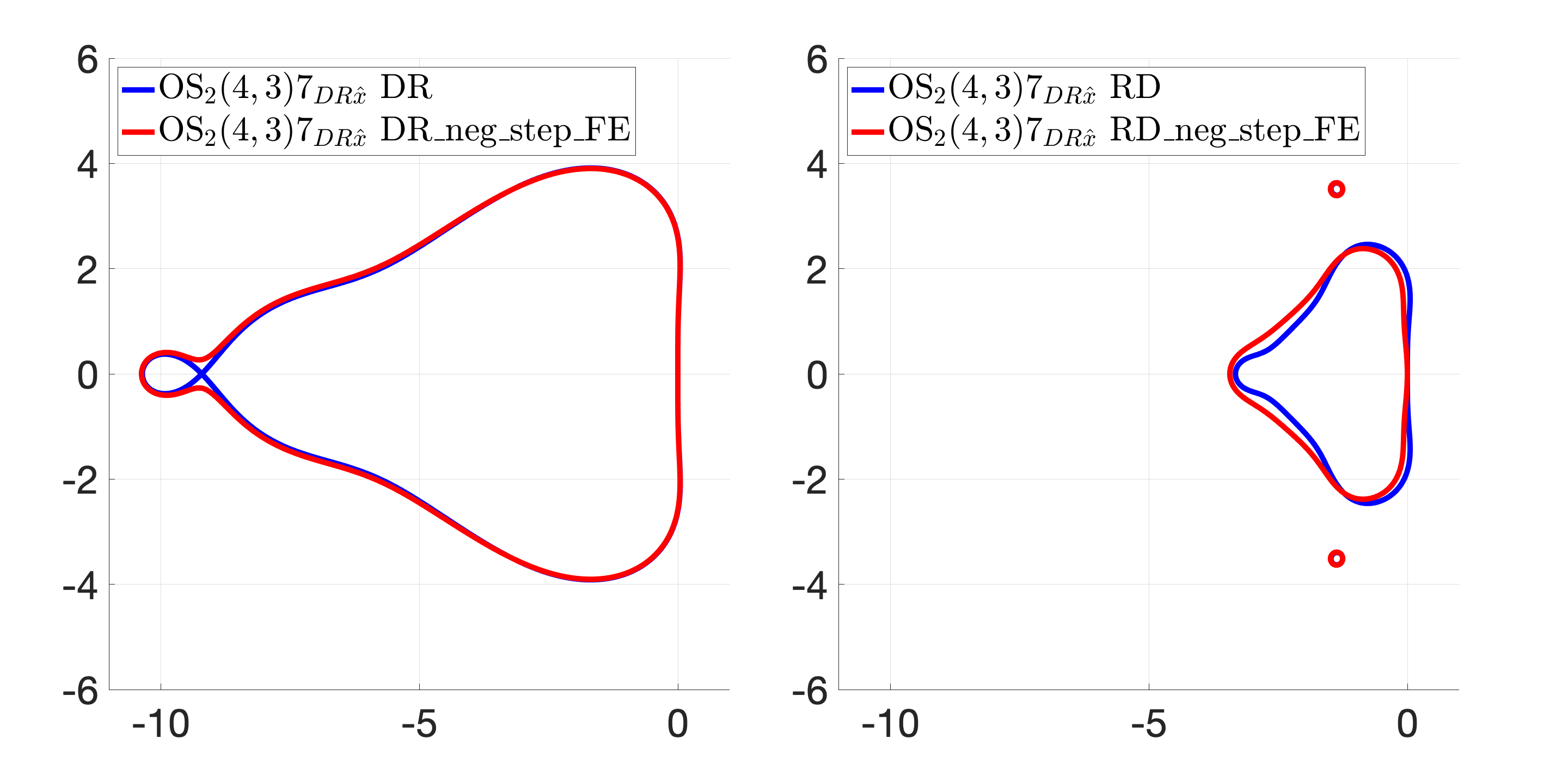}
	\caption{Stability regions of the \osftsDRminx\
      applied to the Niederer benchmark problem with different
      operator orderings and FE applied to both negative sub-integrations.}
  \label{fig:DR_RD_FE_comp2}
\end{figure}

\begin{table}[!htbp]
	\centering
	\begin{tabular}{|c|c|c|c|c|c|c|}
		\hline 
		method & \multicolumn{3}{c|}{DR (FE for both neg)} & \multicolumn{3}{c|}{RD (FE for both neg)} \\ 
		\hline 
		& $\Delta t$ & time & error & $\Delta t$ & time & error  \\ 
		\hline 
		Ruth &  0.0029 &	5752 &	0.021  & 0.0062 & 2678 &	0.041  \\ 
		\hline 
		AKS3 & 0.0031 & 	5338 & 	0.024 & 	0.0031 & 	5290 & 	0.024 \\ 
		\hline 
		\osftsDRminx & \textbf{0.011} & \textbf{1906} & \textbf{0.0414} & --- & --- & ---  \\
		\hline 
	\end{tabular}
	\caption{Efficiency comparison of solving the Niederer benchmark
      problem using the Ruth, AKS3, and \osftsDRminx\
      methods with the DR and RD orderings using forward Euler for
      both backward-in-time sub-integrations.}
    \label{tab:DR_RD_FE_comp}
\end{table}

%
%
%
%
%
%
%



\section{Conclusions}
\label{sec:beat_ruth_conclu}

In this paper, we constructed a new four-stage, third-order, 2-split
operator-splitting method with seven sub-integrations per time step
and improved linear stability and computational efficiency over
existing methods and implementations. The method was designed by
optimizing the linear stability region on the negative real axis while
considering the number and type of sub-integrations. This optimization
represents balancing the tradeoff of stability against computational
expense per step. Moreover, we also proposed two novel implementation
strategies to improve the stability of FSRK methods for a given
problem. First, operator ordering can be chosen to optimize the
stability. Second, replacing an implicit Runge--Kutta sub-integrator
with an explicit Runge--Kutta sub-integrator for backward-in-time
sub-integrations generally improves efficiency and can yield
(sometimes substantial) increases in stability. When low-order
explicit Runge--Kutta sub-integrators are used, the increase in
efficiency may nonetheless offset the increase in error, especially
for stability-constrained problems. For future work, we plan to
further explore high-order operator-splitting methods with an emphasis
on problems where the Jacobians are not simultaneously diagonalizable.


\bibliographystyle{siamplain}
\bibliography{reference}
\end{document}